\documentclass[12pt]{article}
\usepackage[a4paper,margin=1.2in]{geometry}
\usepackage{amssymb}
\usepackage{amsthm}
\usepackage{amsmath}
\usepackage{amsfonts}
\usepackage{graphicx}
\usepackage{float}
\usepackage[all]{xy}
\usepackage{scalerel}

\def\hurw{\mathop{\mbox{\textsl{H}}}\nolimits}
\def\ker{\mathop{\mbox{\textsl{Ker}}}\nolimits}
\def\ann{\mathop{\mbox{\textsl{Ann}}}\nolimits}

\def\exp{\mathop{\mbox{\textsf{exp}}}\nolimits}

\newcommand{\Opa}{\mathfrak{A}}

\newcommand{\N}{\mathbb{N}}

\newcommand{\R}{\mathbb{R}}

\newcommand{\n}{\textbf{n}}
\newcommand{\x}{\mathbf{x}}

\newcommand{\h}{\textbf{h}}

\DeclareMathOperator*{\bcast}{\scalerel*{\odot}{\sum}}

\DeclareMathOperator*{\bbox}{\scalerel*{\boxplus}{\sum}}

\newtheorem{theorem}{Theorem}
\newtheorem{lemma}{Lemma}
\newtheorem{definition}{Definition}
\newtheorem{proposition}{Proposition}

\newtheorem{example}{Example}

\begin{document}
\title{\textbf{Ring of Flows of $k$-dimensional Autonomous Dynamical Systems}}
\author{Ronald Orozco L\'opez}

\newcommand{\Addresses}{{% additional braces for segregating \footnotesize
%  \bigskip
%  \footnotesize

%  R.~Orozco, \textsc{Departamento de Matematicas, Universidad de los Andes, 
%   Carrera 1 N. 18A-12 Bogot\'a, Colombia}\par\nopagebreak
\textit{E-mail address}, R.~Orozco: \texttt{rj.orozco@uniandes.edu.co}
  
}}

\maketitle

\begin{abstract}
We construct a ring of flows where we can decompose autonomous nonlinear dynamical systems into smaller parts, then solve each part and finally put everything together to obtain the exact solution of these systems.
\end{abstract}
{\bf Keywords:} ring Hurwitz, phi-expansion ring, autonomous ring, flow ring\\
{\bf Mathematics Subject Classification:} 13f25, 34a34

\section{Introduction}

Un sistema dinámico es una terna $(T,X,\phi)$, donde $T$ es el conjunto de tiempos, $X\subseteq\R^{n}$ es el espacio de fases y $\phi$ es el mapa diferenciable $\phi:T\times X\rightarrow X$ solución de la ecuación diferencial $U^{\prime}=f(U)$, $f:D\subseteq\R^{n}\rightarrow\R^{n}$, y satisfaciendo las propiedades de grupo uniparámetrico
\begin{enumerate}
    \item $\phi_{0}(x)=x$,
    \item $\phi_{t}\circ\phi_{s}(x)=\phi_{t+s}(x)$.
\end{enumerate}
El mapa $\phi$ se llama flujo del sistema dinámico y cuando $T=\R$, el flujo será llamado completo o si $T=\R^{+}$, el flujo es no completo. Luego la familia de mapas $\{\phi_{t}:t\in\R\}$ es un grupo aditivo si el flujo es completo o un semigrupo en caso contrario. Para cada $x$ fija, $\phi_{t}(x)$ define una curva o trayectoria en $X$ cuando $t$ varía en $T$. Si la ecuación diferencial tiene la forma $U^{\prime}=AU$, con $A$ un matriz $n\times n$, entonces diremos que el sistema dinámico es lineal.

A un sistema dinámico lineal podemos aplicarle el principio de superposición, esto es, podemos descomponer el sistema lineal en partes. Entonces cada parte puede ser resuelta separadamente y luego todo ser recombinado para obtener la respuesta final. Esto no es cierto para sistemas no lineales, pues no existe un principio de superposición para este tipo de sistemas, ya que siempre que las partes de un sistema interfieren, cooperan o compiten, se producen interacciones no lineales. En este artículo construiremos un anillo en donde podamos descomponer un sistema no lineal en sistemas más pequeños para luego resolverlos y posteriormente juntar todo para obtener la solución final. El parámetro tiempo comumente es tomado del conjunto de los números reales, en este artículo dicho parámetro correrá sobre un dominio de integridad $R$ de caracterísica cero. Luego en futuros trabajos será posible usar toda la maquinaria del álgebra en el estudio de soluciones de sistemas dinámicos no lineales.

Este artículo está dividido de la siguiente forma. En la segunda sección introduciremos el anillo diferencial $\hurw_{R}[[x_{1},...,x_{k}]]$ de series de potencias en las variables $x_{1},...,x_{k}$. Seguido, en la tercera sección estudiamos el anillo de $k$ copias del anillo $\hurw_{R}[[x_{1},...,x_{k}]]$ y sobre este anillo es definida la matriz Jacobiana. En la cuarta sección por medio de usar derivadas direccionales son construidos los anillos expansión y autónomos. En la quinta sección es definido el flujo $k$-dimensional sobre dominios de integridad y es mostrado que es posible dotar con estructura de $R$-módulo al flujo de una sistema dinámico cuando este es definido sobre un dominio de integridad. Finalmente, construimos el anillo de funciones generadoras exponenciales de los anillos de la quinta sección. Sobre estos últimos anillos es donde estudiaremos la solución de un sistema dinámico. En particular se darán soluciones exactas de la ecuación de Lotka-Volterra y de la ecuación de Van der Pol.

\section{El anillo $(\hurw_{R}[[x_{1},...,x_{k}]],+,\cdot)$}

En todo este art\'iculo denote $(R,+,\cdot)$ un dominio de integridad de caracter\'istica $0$. Sea
$\N^{k}$ un conjunto de \'indices con orden lexicográfico y denote $\n=(n_{1},...,n_{k})$. Definimos los siguientes conjuntos:

\begin{equation}
\hurw_{R}=\{(a_{\n})_{\n\in\N^{k}}:a_{\n}\in R\}
\end{equation}
en donde los elementos de la sucesi\'on $(a_{\n})$ vienen dados con el orden de $\N^{k}$.

Sean $\textbf{a}=(a_{\n})_{\n\in\N^{k}}$ y $\textbf{b}=(b_{\n})_{\n\in\N^{k}}$ en $\hurw_{R}$. Defina 
la suma en $\hurw_{R}$ como $\textbf{a}+\textbf{b}=(a_{\n}+b_{\n})_{\n\in\N^{k}}$. Denote $\cdot$ 
el producto Hadamard de $\textbf{a}$ y $\textbf{b}$ dado por 
$\textbf{a}\cdot\textbf{b}=(a_{\n}\cdot b_{\n})_{\n\in\N^{k}}$, es decir, el producto en 
$\hurw_{R}$ es definido componente a componente. Luego $(H_{R},+,\cdot)$ es un anillo con elemento
unidad $\textbf{1}=(1,1,1,...)$. Un elemento $\textbf{a}$ es invertible en $\hurw_{R}$ con respecto
al producto $\cdot$ si y s\'olo si $a_{\n}\in R^{*}$ para todo $\n\in\N^{k}$.

Sea $\textbf{h}=(h_{1},...,h_{k})\in\N^{k}$. Ahora denote $*$ el $k$-producto Hurwitz en $\hurw_{R}$ como
\begin{equation}\label{eqn_k_prod_h}
\textbf{a}\ast\textbf{b}=\left(\sum_{\textbf{h}=\textbf{0}}^{\n}\prod_{j=1}^{k}\binom{h_{j}}{n_{j}}a_{\textbf{h}}b_{\textbf{n}-\textbf{h}}\right)_{\textbf{n}\in\N^{k}}.
\end{equation}

Entonces $(\hurw_{R},+,\ast)$ es un anillo con elemento unidad $\textbf{e}=(1,0,0,...)$. Un elemento
$\textbf{a}=(a_{\n})_{\n\in\N^{k}}$ en $\hurw_{R}$ es invertible con respecto a $\ast$ si y s\'olo 
$a_{0}\in R^{*}$. Denote 
$$\hurw_{R}^{*}=\{\textbf{a}^{-1}:\textbf{a}\in\hurw_{R}\}$$ 
el conjunto de elementos invertibles en $(\hurw_{R},+,*)$, donde 
$\textbf{a}^{-1}=\textbf{b}=(b_{\n})_{\n\in\N}$. Entonces

\begin{proposition}
Denote $\vert\h\vert=n_{1}+n_{2}+\cdots+n_{k}$ y $\textbf{0}=(0,...,0)$. Si $\textbf{b}=\textbf{a}^{-1}$, entonces
\begin{eqnarray}
b_{\textbf{0}}&=&a_{\textbf{0}}^{-1}\\
b_{\n}&=&-a_{\textbf{0}}^{-1}\sum_{\vert\h\vert=1}^{\n}\prod_{j=1}^{k}\binom{n_{j}}{h_{j}}a_{\h}b_{\n-\h},\ \ \n\in\N^{k}.
\end{eqnarray}
\end{proposition}
\begin{proof}
Se debe tener $\textbf{a}\ast\textbf{b}=\textbf{e}$. Luego
\begin{equation*}
\sum_{\textbf{h}=\textbf{0}}^{\n}\prod_{j=1}^{k}\binom{h_{j}}{n_{j}}a_{\textbf{h}}b_{\textbf{n}-\textbf{h}}=
\begin{cases}
1 & \text{si }\n=\textbf{0}\\
0 & \text{si }\n\neq\textbf{0}
\end{cases}
\end{equation*}
El resto de la prueba sigue de obtener $b_{\n}$.
\end{proof}

Sea $\x=x_{1}\cdots x_{k}$ un monomial en las indeterminadas $x_{1}$,...,$x_{k}$ y denote 
$\x^{n}=x_{1}^{n_{1}}\cdots x_{k}^{n_{k}}$, 
$\n!=n_{1}!\cdots n_{k}!$. Denote $\hurw_{R}[[\x]]$ el conjunto de series de potencias formales
de la forma $\sum_{\n\in\N^{k}}a_{\n}\frac{\x^{\n}}{\n!}$ con coeficientes en $R$. 
Es claro que $(\hurw_{R}[[\x]],+,\cdot)$ es un anillo con adici\'on y producto de series ordinaria

$$f(\x)+g(\x)=\sum_{\n\in\N^{k}}(a_{\n}+b_{\n})\frac{\x^{\n}}{\n!},$$  
$$f(\x)\cdot g(\x)=\sum_{\n\in\N^{k}}\left(\sum_{\textbf{h}=\textbf{0}}^{\n}\prod_{j=1}^{k}\binom{h_{j}}{n_{j}}a_{\textbf{h}}b_{\textbf{n}-\textbf{h}}\right)\frac{\x^{\n}}{\n!},$$ 

con $f(\x)=\sum_{\n\in\N^{k}}a_{\n}\frac{\x^{\n}}{\n!}$, 
$g(x)=\sum_{\n\in\N^{k}}b_{\n}\frac{\x^{\n}}{\n!}\in\hurw_{R}[[\x]]$. El anillo $\hurw_{R}[[\x]]$ 
ser\'a llamado anillo de Hurwitz de series de potencias multivariadas. Para el caso univariado 
(ve\'ase [4]).

Por otro lado, sea $\rho_{\x}:(\hurw_{R},+,*)\rightarrow(\hurw_{R}[[\x]],+,\cdot)$ un isomorfismo
dado por 
$$\rho_{\x}(\textbf{a})=\rho_{\x}((a_{\n})_{\n\in\N^{k}})=\sum_{\n\in\N^{k}}a_{\n}\frac{\x^{\n}}{\n!}=f(\x).$$

Si $\rho_{\x}(\textbf{a})=f(\x)$ y $\rho_{\x}(\textbf{b})=g(\x)$, entonces

$$\rho_{\x}(\textbf{a}\ast\textbf{b})=\rho_{\x}(\textbf{a})\cdot\rho_{\x}(\textbf{b})$$
y
$$\rho_{\x}(\textbf{a}^{-1})=\frac{1}{\rho_{\x}(\textbf{a})}=(\rho_{\x}(\textbf{a}))^{-1}$$

Denote $\delta_{l}$ un operador $l$-shift sobre $\hurw_{R}$, $1\leq l\leq k,$ definido por 

$$\delta_{l}((a_{\n})_{\n\in\N^{k}})=(a_{\n+e_{l}})_{n\in\N^{k}}$$

donde $e_{l}$ es el vector con 1 en la componente $l$ y 0 en el resto. Tenemos entonces

\begin{proposition}
Los operadores $l$-shift $\delta_{l}$ son derivaciones sobre $\hurw_{R}$, esto es
\begin{enumerate}
\item $\delta_{l}(\textbf{a}+\textbf{b})=\delta_{l}(\textbf{a})+\delta_{l}(\textbf{b})$.
\item $\delta_{l}(\textbf{a}\ast\textbf{b})=\delta_{l}(\textbf{a})\ast\textbf{b}+\textbf{a}\ast\delta_{l}(\textbf{b}).$
\item $\delta_{i}\delta_{j}=\delta_{j}\delta_{i}$, $i\neq j$.
\end{enumerate}
\end{proposition}
\begin{proof}
La prueba de 1 y 3 son directas. Solo mostramos la prueba de 2. Tenemos
\begin{eqnarray*}
\delta_{l}(\textbf{a}\ast\textbf{b})&=&\sum_{\textbf{h}=\textbf{0}}^{\n+e_{l}}\left[\binom{n_{1}}{h_{1}}\cdots\binom{n_{l}+1}{h_{l}}\cdots\binom{n_{k}}{h_{k}}\right]a_{\h}b_{\n+e_{l}-\h}\\
&=&\sum_{\textbf{h}=\textbf{0}}^{\n+e_{l}}\left[\binom{n_{1}}{h_{1}}\cdots\left(\binom{n_{l}}{h_{l}-1}+\binom{n_{l}}{h_{l}}\right)\cdots\binom{n_{k}}{h_{k}}\right]a_{\h}b_{\n+e_{l}-\h}\\
&=&\sum_{\textbf{h}=\textbf{0}}^{\n+e_{l}}\left[\binom{n_{1}}{h_{1}}\cdots\binom{n_{l}}{h_{l}-1}\cdots\binom{n_{k}}{h_{k}}\right]a_{\h}b_{\n+e_{l}-\h}\\
&&+\sum_{\textbf{h}=\textbf{0}}^{\n+e_{l}}\left[\binom{n_{1}}{h_{1}}\cdots\binom{n_{l}}{h_{l}}\cdots\binom{n_{k}}{h_{k}}\right]a_{\h}b_{\n+e_{l}-\h}
\end{eqnarray*}
En la primera sumatoria $h_{l}$ var\'ia entre $0$ y $n_{l}+1$. Luego el coeficiente binomial
$\binom{n_{l}}{h_{l}-1}$ se cancela para $h_{l}=0$. As\'i podemos iniciar $h_{l}$ en 1 y la sumatoria queda desp\'ues de un cambio de variable
$$\sum_{\textbf{h}=\textbf{0}}^{\n}\left[\binom{n_{1}}{h_{1}}\cdots\binom{n_{l}}{h_{l}}\cdots\binom{n_{k}}{h_{k}}\right]a_{\h+e_{l}}b_{\n-\h}=\delta_{l}(\textbf{a})\ast\textbf{b}.$$
De igual manera es mostrado que la segunda suma se puede escribir como
$$\sum_{\textbf{h}=\textbf{0}}^{\n}\left[\binom{n_{1}}{h_{1}}\cdots\binom{n_{l}}{h_{l}}\cdots\binom{n_{k}}{h_{k}}\right]a_{\h}b_{\n+e_{l}-\h}=\textbf{a}\ast\delta_{l}(\textbf{b}).$$
Sumando todo lo anterior obtenemos el resultado deseado.
\end{proof}
Denote $\Delta=\{\delta_{1},\delta_{2},...,\delta_{k}\}$ el conjunto de derivaciones en $\hurw_{R}$.
Entonces el anillo $\hurw_{R}$ es un anillo diferencial o $\Delta$-anillo. (Ve\'ase [?])

Ahora denote $\partial_{l}\equiv\frac{\partial}{\partial x_{l}}$ derivaciones parciales sobre 
$\hurw_{R}[[\x]]$ definidas por $\partial_{l}x_{l}^{n}=nx_{l}^{n-1}$ para $n\geq0$,
$\partial_{l}x_{i}=0$ para todo $l\neq i$ y

$$\partial_{l}\left(\sum_{\n\in\N^{k}}a_{\n}\frac{\x^{\n}}{\n!}\right)=\sum_{\n\in\N^{k}}a_{\n}\frac{\partial_{l}\x^{\n}}{\n!}=\sum_{\n\in\N^{k}}a_{\n+e_{l}}\frac{\x^{\n}}{\n!}.$$

Entonces con $\nabla=\{\partial_{1},\partial_{2},...,\partial_{k}\}$ el anillo $\hurw_{R}[[\x]]$
viene a ser un anillo diferencial o $\nabla$-anillo. Adem\'as

\begin{equation}\label{relation_delta_rho}
\partial_{l}\rho_{\x}(\textbf{a})=\rho_{\x}(\delta_{l}\textbf{a}),
\end{equation}

es decir, el siguiente diagrama conmuta

\begin{equation}
\xymatrix{
 \hurw_{R} \ar[d]^{\rho_{\x}} \ar[r]^{\delta_{l}} & \hurw_{R} \ar[d]^{\rho_{\x}}\\
 \hurw_{R}[[\x]] \ar[r]^{\partial_{l}} & \hurw_{R}[[\x]] 
}
\end{equation}
y el isomorfismo $\rho_{\x}$ es un isomorfismo diferencial.

Denote $C_{l}(\hurw_{R})$ el conjunto de constantes de la derivaci\'on $\delta_{l}$, esto es
\begin{equation}
C_{l}(\hurw_{R})=\{\textbf{a}\in\hurw_{R}:\delta_{l}(\textbf{a})=(0,0,...)\}
\end{equation}
y denote $C_{l}(\hurw_{R}[[\x]])$ el conjunto de constantes de la derivaci\'on $\partial_{l}$, 
esto es
\begin{equation}
C_{l}(\hurw_{R}[[\x]])=\{\rho_{\x}(\textbf{a})\in\hurw_{R}[[\x]]:\partial_{l}(\rho_{\x}(\textbf{a}))=0\}
\end{equation}

Entonces por (\ref{relation_delta_rho}), 
$$\partial_{l}\rho_{\x}(\textbf{a})=\rho_{\x}(\delta_{l}\textbf{a})=\rho_{\x}(0,0,...)=0$$
para todo $\textbf{a}\in C_{l}(\hurw_{R})$ y $\rho_{\x}(\textbf{a})$ es una constante para 
$\partial_{l}$. Es decir, $\rho_{\x}(\textbf{a})\in C_{l}(\hurw_{R}[[\x]])$ y
$$\rho_{\x}(C_{l}(\hurw_{R}))=C_{l}(\hurw_{R}[[\x]]).$$

Finalizamos mostrando los ideales en $\hurw_{R}[[\x]]$. Sea $I$ un ideal en $R$ y sea $\epsilon=(\epsilon_{1},\epsilon_{2},...,\epsilon_{k})$
con $\epsilon_{i}=0,1$ y denote $\x^{\epsilon}=x_{1}^{\epsilon_{1}}x_{2}^{\epsilon_{2}}\cdots x_{k}^{\epsilon_{k}}$ y suponga que $\x^{\epsilon}\neq1$. Entonces
\begin{equation}
I+\left\langle\x^{\epsilon}\right\rangle=\left\{\sum_{\n\in\N^{k}}a_{\n}\frac{\x^{\n}}{\n!}:a_{\textbf{0}}\in I\right\}
\end{equation}
es un ideal en $\hurw_{R}[[\x]]$. Otro ideal en $\hurw_{R}[[\x]]$ es de la forma
\begin{equation}
\hurw_{I}[[\x]]=\left\{\sum_{\n\in\N^{k}}a_{\n}\frac{\x^{\n}}{\n!}:a_{\textbf{n}}\in I\right\}
\end{equation}

\section{El anillo $\hurw_{R}^{k}[[x_{1},...,x_{k}]]$}

Denote $\hurw_{R}^{k}$ el producto directo de $k$ copias del anillo de Hurwitz $\hurw_{R}$ en donde
$$\mathfrak{a}+\mathfrak{b}=(\textbf{a}_{1},...,\textbf{a}_{k})+(\textbf{b}_{1},...,\textbf{b}_{k})=(\textbf{a}_{1}+\textbf{b}_{1},...,\textbf{a}_{k}+\textbf{b}_{k})$$
y
$$\mathfrak{a}\ast\mathfrak{b}=(\textbf{a}_{1},...,\textbf{a}_{k})\ast(\textbf{b}_{1},...,\textbf{b}_{k})=(\textbf{a}_{1}\ast\textbf{b}_{1},...,\textbf{a}_{k}\ast\textbf{b}_{k})$$
para todo $\mathfrak{a},\mathfrak{b}\in\hurw_{R}^{k}$. Ahora denote $\hurw_{R}^{k}[[\x]]$ el producto
directo de $k$ copias del anillo de series exponenciales multivariadas $\hurw_{R}[[\x]]$ en donde
\begin{eqnarray*}
F(\x)+G(\x)&=&(f_{1}(\x),...,f_{k}(\x))+(g_{1}(\x),...,g_{k}(\x))\\
&=&(f_{1}(\x)+g_{1}(\x),...,f_{k}(\x)+g_{k}(\x))
\end{eqnarray*}
y
\begin{eqnarray*}
F(\x)\cdot G(\x)&=&(f_{1}(\x),...,f_{k}(\x))\cdot(g_{1}(\x),...,g_{k}(\x))\\
&=&(f_{1}(\x)\cdot g_{1}(\x),...,f_{k}(\x)\cdot g_{k}(\x))
\end{eqnarray*}
para todo $F(\x),G(\x)\in\hurw_{R}^{k}[[\x]]$.

Es f\'acil notar que ambos anillos $\hurw_{R}^{k}$ y $\hurw_{R}^{k}[[\x]]$ son anillos con divisores
de cero. Ahora si definimos el mapa $\rho_{\x}$ sobre el anillo $\hurw_{R}^{k}$ como
$$\rho_{\x}(\mathfrak{a})=(\rho_{\x}(\textbf{a}_{1}),...,\rho_{\x}(\textbf{a}_{k}))$$
entonces $\rho_{\x}$ es un ismorfismo de $\hurw_{R}^{k}$ en $\hurw_{R}^{k}[[\x]]$ con
\begin{eqnarray*}
\rho_{\x}(\mathfrak{a}+\mathfrak{b})&=&\rho_{\x}(\mathfrak{a})+\rho_{\x}(\mathfrak{b})\\
\rho_{\x}(\mathfrak{a}\ast\mathfrak{b})&=&\rho_{\x}(\mathfrak{a})\cdot\rho_{\x}(\mathfrak{b})
\end{eqnarray*} 

El conjunto de todas las matrices $m\times n$ con entradas de $\hurw_{R}$ ser\'a denotado por
$M_{m\times n}(\hurw_{R})$. Con $\mathbf{a}_{i,j}$ denotaremos la entrada $i,j$-\'esima de la matriz
$\mathrm{A}=(\mathbf{a}_{i,j})_{i,j=1}^{m,n}$.

\begin{definition}
Para matrices $\mathrm{A}=(\mathbf{a}_{i,j})_{i,j=1}^{m,n},\mathrm{B}=(\mathbf{b}_{i,j})_{i,j=1}^{m,n}\in M_{m\times n}(\hurw_{R})$ definimos la suma $\mathrm{A}+\mathrm{B}$ en $M_{m\times n}(\hurw_{R})$ como
\begin{equation}
\mathrm{A}+\mathrm{B}=(\mathbf{a}_{i,j}+\mathbf{b}_{i,j})_{i,j=1}^{m,n}.
\end{equation}
Para matrices $\mathrm{A}=(\mathbf{a}_{i,j})_{i,j=1}^{m,n}\in M_{m\times n}(\hurw_{R})$ y $\mathrm{B}=(\mathbf{a}_{i,j})_{i,j=1}^{n,p}\in M_{n\times p}(\hurw_{R})$ definimos el producto $\mathrm{A}\mathrm{B}$ en $M_{m\times p}(\hurw_{R})$ como
\begin{equation}
\mathrm{A}\mathrm{B}=\left(\sum_{k=1}^{n}\textbf{a}_{ik}\ast\textbf{b}_{kj}\right)_{i,j=1}^{m,n}.
\end{equation}
La multiplicaci\'on escalar es definida como 
\begin{equation}
\textbf{r}\ast\mathrm{A}=(\textbf{r}\ast\mathbf{a}_{i,j})_{i,j=1}^{m,n}
\end{equation}
para todo $\mathbf{r}\in\hurw_{R}$ y toda matriz $\mathrm{A}\in M_{m\times n}(\hurw_{R})$. Por
\'ultimo, denote $\mathrm{row}_{i}(A)$ la $i$-\'esima fila de $\mathrm{A}\in M_{m\times n}(\hurw_{R})$. Entonces
\begin{equation}
\mathfrak{a}\ast\mathrm{A}=\mathrm{A}\ast\mathfrak{a}=(\mathbf{a}_{1}\ast\mathrm{row}_{1}(A),...,\mathbf{a}_{n}\ast\mathrm{row}_{n}(A))
\end{equation}
para $\mathfrak{a}=(\mathbf{a}_{1},...,\mathbf{a}_{n})\in\hurw_{R}^{n}$.
\end{definition}

Es muy f\'acil probar el siguiente resultado

\begin{proposition}
El conjunto de matrices $M_{m\times n}(\hurw_{R})$ es un $\hurw_{R}$-m\'odulo y el conjunto de matrices $M_{n\times n}(\hurw_{R})$ es un $\hurw_{R}$-\'algebra de matrices.
\end{proposition}

Si definimos el mapa $\rho_{\x}$ sobre el el conjunto de matrices $M_{m\times n}(\hurw_{R})$,
entonces 
\begin{equation}
\rho_{\x}[M_{m\times n}(\hurw_{R})]=M_{m\times n}(\rho_{\x}(\hurw_{R}))=M_{m\times n}(\hurw_{R}[[\x]])
\end{equation}
y
\begin{eqnarray}
\rho_{\x}(\mathrm{A}+\mathrm{B})&=&\rho_{\x}(\mathrm{A})+\rho_{\x}(\mathrm{B})\\
\rho_{\x}(\mathrm{A}\mathrm{B})&=&\rho_{\x}(\mathrm{A})\rho_{\x}(\mathrm{B})\\
\rho_{\x}(\mathbf{r}\ast\mathrm{A})&=&\rho_{\x}(\mathbf{r})\rho_{\x}(\mathrm{A})\\
\rho_{\x}(\mathfrak{a}\ast\mathrm{A})&=&\rho_{\x}(\mathfrak{a})\cdot\rho_{\x}(\mathrm{A})\nonumber\\
&=&(\rho_{\x}(\mathbf{a}_{1})\cdot\mathrm{row}(\rho_{\x}(A)),...,\rho_{\x}(\mathbf{a}_{n})\cdot\mathrm{row}(\rho_{\x}(A)))
\end{eqnarray}
y como claramente $M_{m\times n}(\hurw_{R}[[\x]])$ es un $\hurw_{R}[[\x]]$-m\'odulo y 
$M_{n\times n}(\hurw_{R}[[\x]])$ es un $\hurw_{R}[[\x]]$-\'algebra, entones $\rho_{\x}$ es un isomorfismo.

Ahora definiremos el an\'alogo de la matriz Jacobiana para elementos en los anillos $\hurw_{R}^{k}$
y $\hurw_{R}^{k}[[\x]]$
\begin{definition}
Tome $\mathfrak{a}=(\textbf{a}_{1},...,\textbf{a}_{k})\in\hurw_{R}^{k}$. La matrix Jacobiana de $\mathfrak{a}$ se define como la matriz $k\times k$ en $M_{k\times k}(\hurw_{R})$
\begin{equation}
\delta(\mathfrak{a})=
\left[
\begin{array}{ccc}
\delta_{1}(\textbf{a}_{1})&\cdots&\delta_{k}(\textbf{a}_{1})\\
\vdots&\ddots&\vdots\\
\delta_{1}(\textbf{a}_{k})&\cdots&\delta_{k}(\textbf{a}_{k})
\end{array}
\right]
\end{equation}
Tome $F(\x)=(f_{1}(\x),...,f_{k}(\x))\in\hurw_{R}^{k}[[\x]]$. La matrix Jacobiana de $F(\x)$ se define como la matriz $k\times k$ en $M_{k\times k}(\hurw_{R}[[\x]])$
\begin{equation}
\frac{\partial F(\x)}{\partial\x}=
\left[
\begin{array}{ccc}
\partial_{1}(f_{1}(\x))&\cdots&\partial_{k}(f_{1}(\x))\\
\vdots&\ddots&\vdots\\
\partial_{1}(f_{k}(\x))&\cdots&\partial_{k}(f_{k}(\x))
\end{array}
\right]
\end{equation}
\end{definition}

Luego
\begin{equation}
\frac{\partial}{\partial\x}\rho_{\x}(\mathfrak{a})=\rho_{\x}(\delta(\mathfrak{a}))
\end{equation}

Ahora definimos una derivada sobre el anillo $\hurw_{R}^{k}$ 
\begin{proposition}
Fije $\mathfrak{h}$ en $\hurw_{R}^{k}$. Defina el mapa $\mathfrak{D}_{\mathfrak{h}}:\hurw_{R}^{k}\rightarrow\hurw_{R}^{k}$ como $\mathfrak{D}_{\mathfrak{h}}(\mathfrak{a})=\delta(\mathfrak{a})\mathfrak{h}^{\top}$ para cada
$\mathfrak{a}$ en $\hurw_{R}^{k}$, en donde $\mathfrak{h}^{\top}$ es el vector columna de 
$\mathfrak{h}$. Entonces $\mathfrak{D}_{\mathfrak{h}}$ es una derivaci\'on sobre el anillo 
$\hurw_{R}^{k}$. 
\end{proposition}
\begin{proof}
Tome $\mathfrak{a},\mathfrak{b}$ en $\hurw_{R}^{k}$. Como la matriz Jacobiana $\delta$ es lineal, entonces
\begin{eqnarray*}
\mathfrak{D}_{\mathfrak{h}}(\mathfrak{a}+\mathfrak{b})&=&\delta(\mathfrak{a}+\mathfrak{b})\mathfrak{h}^{\top}\\
&=&[\delta(\mathfrak{a})+\delta(\mathfrak{b})]\mathfrak{h}^{\top}\\
&=&\delta(\mathfrak{a})\mathfrak{h}^{\top}+\delta(\mathfrak{b})\mathfrak{h}^{\top}\\
&=&\mathfrak{D}_{\mathfrak{h}}(\mathfrak{a})+\mathfrak{D}_{\mathfrak{h}}(\mathfrak{b}).
\end{eqnarray*}
Ahora probaremos la regla de Leibniz. Tenemos
\begin{eqnarray*}
\mathfrak{D}_{\mathfrak{h}}(\mathfrak{a}\ast\mathfrak{b})&=&\delta(\mathfrak{a}\ast\mathfrak{b})\mathfrak{h}^{\top}\\
&=&\left[
\begin{array}{ccc}
\delta_{1}(\textbf{a}_{1}\ast\textbf{b}_{1})&\cdots&\delta_{k}(\textbf{a}_{1}\ast\textbf{b}_{1})\\
\vdots&\ddots&\vdots\\
\delta_{1}(\textbf{a}_{k}\ast\textbf{b}_{k})&\cdots&\delta_{k}(\textbf{a}_{k}\ast\textbf{b}_{k})
\end{array}
\right]\mathfrak{h}^{\top}\\
&=&\left[
\begin{array}{ccc}
\delta_{1}(\textbf{a}_{1})\ast\textbf{b}_{1}+\textbf{a}_{1}\ast\delta_{1}(\textbf{b}_{1})&\cdots&\delta_{k}(\textbf{a}_{1})\ast\textbf{b}_{1}+\textbf{a}_{1}\ast\delta_{k}(\textbf{b}_{1})\\
\vdots&\ddots&\vdots\\
\delta_{1}(\textbf{a}_{k})\ast\textbf{b}_{k}+\textbf{a}_{k}\ast\delta_{1}(\textbf{b}_{k})&\cdots&\delta_{k}(\textbf{a}_{k})\ast\textbf{b}_{k}+\textbf{a}_{k}\ast\delta_{k}(\textbf{b}_{k})
\end{array}
\right]\mathfrak{h}^{\top}\\
&=&\left[
\begin{array}{ccc}
\delta_{1}(\textbf{a}_{1})\ast\textbf{b}_{1}&\cdots&\delta_{k}(\textbf{a}_{1})\ast\textbf{b}_{1}\\
\vdots&\ddots&\vdots\\
\delta_{1}(\textbf{a}_{k})\ast\textbf{b}_{k}&\cdots&\delta_{k}(\textbf{a}_{k})\ast\textbf{b}_{k}
\end{array}
\right]\mathfrak{h}^{\top}\\
&&+\left[
\begin{array}{ccc}
\textbf{a}_{1}\ast\delta_{1}(\textbf{b}_{1})&\cdots&\textbf{a}_{1}\ast\delta_{k}(\textbf{b}_{1})\\
\vdots&\ddots&\vdots\\
\textbf{a}_{k}\ast\delta_{1}(\textbf{b}_{k})&\cdots&\textbf{a}_{k}\ast\delta_{k}(\textbf{b}_{k})
\end{array}
\right]\mathfrak{h}^{\top}\\
&=&[\delta(\mathfrak{a})\ast\mathfrak{b}]\mathfrak{h}^{\top}+[\mathfrak{a}\ast\delta(\mathfrak{b})]\mathfrak{h}^{\top}\\
&=&[\delta(\mathfrak{a})\mathfrak{h}^{\top}]\ast\mathfrak{b}+\mathfrak{a}\ast[\delta(\mathfrak{b})\mathfrak{h}^{\top}]
=\mathfrak{D}_{\mathfrak{h}}(\mathfrak{a})\ast\mathfrak{b}+\mathfrak{a}\ast\mathfrak{D}_{\mathfrak{h}}(\mathfrak{b})
\end{eqnarray*}
\end{proof}

\begin{proposition}
De manera an\'aloga fije $\phi$ en $\hurw_{R}^{k}[[\x]]$. Defina el mapa
$\mathfrak{D}_{\phi}:\hurw_{R}^{k}[[\x]]\rightarrow\hurw_{R}^{k}[[\x]]$ como $\mathfrak{D}_{\phi}(F(\x))=\frac{\partial F(\x)}{\partial\x}\phi^{\top}$ para cada $F(\x)$ en $\hurw_{R}^{k}[[\x]]$. 
Entonces $\mathfrak{D}_{\mathfrak{h}}$ es una derivaci\'on sobre el anillo 
$\hurw_{R}^{k}[[\x]]$.
\end{proposition}
\begin{proof}
Suponga que $\mathfrak{D}_{\mathfrak{h}}$ es una derivaci\'on sobre $\hurw_{R}^{k}$. Entonces
\begin{equation}
\rho_{\x}[\mathfrak{D}_{\mathfrak{h}}(\mathfrak{a})]=\rho_{\x}(\delta(\mathfrak{a})\mathfrak{h}^{\top})=\rho_{\x}(\delta(\mathfrak{a}))\rho_{\x}(\mathfrak{h}^{\top})=\frac{\partial\rho_{\x}(\mathfrak{a})}{\partial\x}\rho_{\x}(\mathfrak{h}^{\top})
\end{equation}
Defina $\mathfrak{D}_{\phi}=\rho_{\x}(\mathfrak{D}_{\mathfrak{h}})$ con 
$\phi=\rho_{\x}(\mathfrak{h})$. Entonces
\begin{eqnarray*}
\mathfrak{D}_{\phi}(F(\x)+G(\x))&=&\frac{\partial[\rho_{\x}(\mathfrak{a})+\rho_{\x}(\mathfrak{b})]}{\partial\x}\rho_{\x}(\mathfrak{h}^{\top})\\
&=&\frac{\partial\rho_{\x}(\mathfrak{a})}{\partial\x}\rho_{\x}(\mathfrak{h}^{\top})+\frac{\partial\rho_{\x}(\mathfrak{b})}{\partial\x}\rho_{\x}(\mathfrak{h}^{\top})\\
&=&\mathfrak{D}_{\phi}(F(\x))+\mathfrak{D}_{\phi}(G(\x))
\end{eqnarray*}
y
\begin{eqnarray*}
\mathfrak{D}_{\phi}(F(\x)\cdot G(\x))&=&\frac{\partial[\rho_{\x}(\mathfrak{a})\cdot\rho_{\x}(\mathfrak{b})]}{\partial\x}\rho_{\x}(\mathfrak{h}^{\top})\\
&=&\frac{\partial\rho_{\x}(\mathfrak{a}\ast\mathfrak{b})}{\partial\x}\rho_{\x}(\mathfrak{h}^{\top})\\
&=&\rho_{\x}(\delta(\mathfrak{a}\ast\mathfrak{b}))\rho_{\x}(\mathfrak{h}^{\top})\\
&=&\rho_{\x}(\delta(\mathfrak{a})\ast\mathfrak{b}+\mathfrak{a}\ast\delta(\mathfrak{b}))\rho_{\x}(\mathfrak{h}^{\top})\\
&=&\rho_{\x}(\delta(\mathfrak{a})\ast\mathfrak{b})\rho_{\x}(\mathfrak{h}^{\top})+\rho_{\x}(\mathfrak{a}\ast\delta(\mathfrak{b}))\rho_{\x}(\mathfrak{h}^{\top})\\
&=&[\rho_{\x}(\delta(\mathfrak{a}))\cdot\rho_{\x}(\mathfrak{b})]\rho_{\x}(\mathfrak{h}^{\top})+[\rho_{\x}(\mathfrak{a})\cdot\rho_{\x}(\delta(\mathfrak{b}))]\rho_{\x}(\mathfrak{h}^{\top})\\
&=&\left[\frac{\partial\rho_{\x}(\mathfrak{a})}{\partial\x}\rho_{\x}(\mathfrak{h}^{\top})\right]\cdot\rho_{\x}(\mathfrak{b})+\rho_{\x}(\mathfrak{a})\cdot\left[\frac{\partial\rho_{\x}(\mathfrak{b})}{\partial\x}\rho_{\x}(\mathfrak{h}^{\top})\right]\\
&=&\mathfrak{D}_{\phi}(F(\x))\cdot G(\x)+F(\x)\cdot\mathfrak{D}_{\phi}(G(\x)).
\end{eqnarray*}
\end{proof}

Los ideales en $\hurw_{R}^{k}[[\x]]$ son de la forma $\prod_{i}^{n}J_{i}$, donde
cada $J_{i}$ son o de la forma $I+\left\langle\x^{\epsilon}\right\rangle$ o de la
forma $H_{I}[[\x]]$, con $I$ un ideal en $R$.

\section{Anillos $\phi$-expansión y aut\'onomo}

\begin{definition}
Fije un $\phi$ en $\hurw_{R}^{k}[[\x]]$. Definimos el operador aut\'onomo actuando sobre 
$\hurw_{R}^{k}[[\x]]$ como el mapa $\Opa_{\phi}:\hurw_{R}^{k}[[\x]]\rightarrow\hurw_{S}$, 
$S=\hurw_{R}^{k}[[\x]]$, definido por
\begin{equation}\label{eqn_oper_auto}
\Opa_{\phi}(F(\x))=(\mathfrak{D}_{\phi}^{n}(F(\x)))_{n\in\N}
\end{equation}
en donde $\mathfrak{D}_{\phi}^{0}(F(\x))=F(\x)$ y $\mathfrak{D}_{\phi}^{n+1}(F(\x))=\mathfrak{D}_{\phi}[\mathfrak{D}_{\phi}^{n}(F(\x))]$
\end{definition}

Para un $\phi$ fijo el conjunto $\Opa_{\phi}(\hurw_{R}^{k}[[\x]])$ con las operaciones $+$ y $\ast$ es un anillo conmutativo con unidad. El anillo $\Opa_{\phi}(\hurw_{R}^{k}[[\x]])$ ser\'a llamado \textbf{anillo} 
$\phi$-\textbf{expansi\'on} del anillo $\hurw_{R}^{k}[[\x]]$. Como 
$\mathfrak{D}_{\phi}(aF(\x))=a\mathfrak{D}_{\phi}(F(\x))$, entonces 
$\Opa_{\phi}(aF(\x))=a\Opa_{\phi}(F(\x))$ y $\Opa_{\phi}(\hurw_{R}^{k}[[\x]])$ es un $R$-\'algebra.

\begin{theorem}
El mapa $\Opa_{\phi}$ es un isomorfismo de anillos.
\end{theorem}
\begin{proof}
Tome $\Opa_{\phi}(F(\x)),\Opa_{\phi}(G(\x))$ en $\Opa_{\phi}(\hurw_{R}^{k}[[\x]])$. Entonces
\begin{eqnarray*}
\Opa_{\phi}(F(\x))+\Opa_{\phi}(G(\x))&=&(\mathfrak{D}_{\phi}^{n}(F(\x)))_{n\in\N}+(\mathfrak{D}_{\phi}^{n}(G(\x)))_{n\in\N}\\
&=&(\mathfrak{D}_{\phi}^{n}(F(\x)+G(\x)))_{n\in\N}\\
&=&\Opa_{\phi}(F(\x)+G(\x))
\end{eqnarray*}
y
\begin{eqnarray*}
\Opa_{\phi}(F(\x))\ast\Opa_{\phi}(G(\x))&=&(\mathfrak{D}_{\phi}^{n}(F(\x)))_{n\in\N}\ast(\mathfrak{D}_{\phi}^{n}(G(\x)))_{n\in\N}\\
&=&\left(\sum_{k=0}^{n}\binom{n}{k}\mathfrak{D}_{\phi}^{k}(F(\x))\cdot\mathfrak{D}_{\phi}^{n-k}(G(\x))\right)_{n\in\N}\\
&=&(\mathfrak{D}_{\phi}^{n}(F(\x)\cdot G(\x)))_{n\in\N}\\
&=&\Opa_{\phi}(F(\x)\cdot G(\x))
\end{eqnarray*}
Luego $\Opa_{\phi}(F(\x))+\Opa_{\phi}(G(\x))$ y $\Opa_{\phi}(F(\x))\ast\Opa_{\phi}(G(\x))$ est\'an en
$\Opa_{\phi}(\hurw_{R}^{k}[[\x]])$.
\end{proof}

Ahora construiremos el anillo aut\'onomo y para ello fijaremos $F(\x)$ en $\hurw_{R}^{k}[[\x]]$.
Primero encontraremos un resultado sobre el mapa $\Opa_{a\phi}$ para cada
$a\in R$. En [?] fue definido el anillo $(\exp(R),\ast,\cdot)$ de las expansiones 
$\exp(a)=(1,a,a^{2},a^{3},...)$ para $a\neq0$ y $\exp(0)=(1,0,0,0,...)$, en donde el producto $\cdot$
es definido componente a componente. Luego

\begin{proposition}\label{prop_exp}
Tome $\phi$ en $\hurw_{R}^{k}[[\x]]$ y $a\in R$. Entonces
\begin{equation}
\exp(a)\Opa_{\phi}(F(\x))=\Opa_{a\phi}(F(\x))
\end{equation}
para cada $F(\x)\in\hurw_{R}^{k}[[\x]]$.
\end{proposition}
\begin{proof}
Para $a\neq0$ es cierto que $\mathfrak{D}_{a\phi}(F(\x))=a\mathfrak{D}_{\phi}(F(\x))$. Luego 
$\mathfrak{D}_{a\phi}^{n}(F(\x))=a\mathfrak{D}_{\phi}^{n}(F(\x))$ y esto implica que
\begin{eqnarray*}
\Opa_{a\phi}(F(\x))&=&(a^{n}\mathfrak{D}_{\phi}^{n}(F(\x)))_{n\in\N}\\
&=&\exp(a)\Opa_{\phi}(F(\x))
\end{eqnarray*}
Cuando $a=0$,
\begin{eqnarray*}
\Opa_{0\phi}(F(\x))=\Opa_{0}(F(\x))=(F(\x),0,0,...)=\exp(0)\Opa_{\phi}(F(\x))
\end{eqnarray*}
\end{proof}

\begin{definition}
Fije $F(\x)$. Definimos la suma $\boxplus$ en $\Opa_{S}(F(\x))$, $S=\hurw_{R}^{k}[[\x]]$, de la
siguiente manera
\begin{equation}
\Opa_{\phi}(F(\x))\boxplus\Opa_{\psi}(F(\x))=(B_{n}(\phi,\psi)(F(\x)))_{n\in\N}
\end{equation}
con
\begin{eqnarray}
B_{0}(\phi,\psi)(F(\x))&=&F(\x)\\
B_{1}(\phi,\psi)(F(\x))&=&\mathfrak{D}_{\phi}(F(\x))+\mathfrak{D}_{\psi}(F(\x))\\
B_{n}(\phi,\psi)(F(\x))&=&\mathfrak{D}_{\phi}^{n}(F(\x))+\mathfrak{D}_{\psi}^{n}(F(\x))+H_{n}(\phi,\psi),\ \ \ n\geq2
\end{eqnarray}
en donde $H_{n}(f,g)$ satisface
\begin{equation}\label{eqn_Hn}
H_{n+1}(\phi,\psi)=\left[\frac{\partial}{\partial\x}\mathfrak{D}_{\psi}^{n}(F(\x))\right]\phi+\left[\frac{\partial}{\partial\x}\mathfrak{D}_{\phi}^{n}(F(\x))\right]\psi+\left[\frac{\partial}{\partial\x}H_{n}(\phi,\psi)\right](\phi+\psi)
\end{equation}
\end{definition}

\begin{lemma}\label{lemma_ope_+}
Para todo $\phi,\psi\in\hurw_{R}^{k}[[\x]]$ se tiene
\begin{equation}
\Opa_{\phi}(F(\x))\boxplus\Opa_{\psi}(F(\x))=\Opa_{\phi+\psi}(F(\x))
\end{equation}
\end{lemma}
\begin{proof}
Calculando la $n$-\'esima derivada direccional de $F(\x)$ en la direcci\'on $\phi+\psi$ obtenemos 
\begin{equation}
\mathfrak{D}_{\phi+\psi}^{n}(F(\x))=\mathfrak{D}_{\phi}^{n}(F(\x))+\mathfrak{D}_{\psi}^{n}(F(\x))+H_{n}(\phi,\psi).
\end{equation}
Si $H_{n}(\phi+\psi)$ satisface (\ref{eqn_Hn}), tenemos
\begin{eqnarray*}
\mathfrak{D}_{\phi+\psi}^{n+1}(F(\x))&=&\left[\frac{\partial}{\partial\x}\mathfrak{D}_{\phi+\psi}^{n}(F(\x))\right](\phi+\psi)\\
&=&\frac{\partial}{\partial\x}[\mathfrak{D}_{\phi}^{n}(F(\x))+\mathfrak{D}_{\psi}^{n}(F(\x))+H_{n}(\phi,\psi)](\phi+\psi)\\
&=&\left[\frac{\partial}{\partial\x}\mathfrak{D}_{\phi}^{n}(F(\x))\right]\phi+\left[\frac{\partial}{\partial\x}\mathfrak{D}_{\psi}^{n}(F(\x))\right]\psi+\left[\frac{\partial}{\partial\x}\mathfrak{D}_{\psi}^{n}(F(\x))\right]\phi\\
&&+\left[\frac{\partial}{\partial\x}\mathfrak{D}_{\phi}^{n}(F(\x))\right]\psi+\left[\frac{\partial}{\partial\x}H_{n}(\phi,\psi)\right](\phi+\psi)\\
&=&\mathfrak{D}_{\phi}^{n+1}(F(\x))+\mathfrak{D}_{\psi}^{n+1}(F(\x))+H_{n+1}(\phi,\psi)
\end{eqnarray*}
Luego es seguido el estamento.
\end{proof}

\begin{theorem}
Fije $F(\x)$. $(\Opa_{S}(F(\x)),\boxplus)$ con $S=\hurw_{R}^{k}[[\x]]$ es un grupo conmutativo.
\end{theorem}
\begin{proof}
Por el Lema \ref{lemma_ope_+}, $\Opa_{\phi}(F(\x))\boxplus\Opa_{\psi}(F(\x))\in\Opa_{S}(F(\x))$.
Ahora tome $\textbf{0}\in\hurw_{R}^{k}[[\x]]$. Entonces es f\'acil ver que
$\Opa_{\textbf{0}}(F(\x))=(F(\x),\textbf{0},\textbf{0},...)=\textbf{0}_{\Opa}$ y por el Lema anterior
\begin{eqnarray*}
\Opa_{\phi}(F(\x))\boxplus\textbf{0}_{\Opa}&=&\Opa_{\phi}(F(\x))\boxplus\Opa_{\textbf{0}}(F(\x))\\
&=&\Opa_{\phi+\textbf{0}}(F(\x))\\
&=&\Opa_{\phi}(F(\x))
\end{eqnarray*}
luego $\textbf{0}_{\Opa}$ es el elemento neutro en $(\Opa_{S}(F(\x)),\boxplus)$. Por la proposici\'on 
anterior es f\'acil mostrar que $\exp(-1)\Opa_{\phi}(F(\x))$ es el inverso aditivo de 
$\Opa_{\phi}(F(\x))$. Finalmente la asociatividad y conmutatividad de $(\Opa_{S}(F(\x)),\boxplus)$ siguen de las propias de $(\hurw_{R}^{k}[[\x]],+)$.
\end{proof}

\begin{theorem}\label{theo_exp_modulo}
$(\Opa_{S}(F(\x)),\boxplus)$ es un $\exp(R)$-m\'odulo.
\end{theorem}
\begin{proof}
De la Proposici\'on \ref{prop_exp} tenemos
\begin{eqnarray*}
(\exp(a)\exp(b))\Opa_{\phi}(F(\x))&=&\exp(a)[\exp(b)\Opa_{\phi}(F(\x))]\\
\exp(a)[\Opa_{\phi}(F(\x))\boxplus\Opa_{\psi}(F(\x))]&=&\exp(a)\Opa_{\phi}(F(\x))\boxplus\exp(a)\Opa_{\psi}(F(\x))\\
(\exp(a)\ast\exp(b))\Opa_{\phi}(F(\x))&=&\exp(a)\Opa_{\phi}(F(\x))\boxplus\exp(b)\Opa_{\phi}(F(\x))\\
\exp(1)\Opa_{\phi}(F(\x))&=&\Opa_{\phi}(F(\x))
\end{eqnarray*}
para todo $a,b\in R$ y todo $\phi,\psi\in\hurw_{R}^{k}[[\x]]$.
\end{proof}

\begin{definition}
Definimos el producto $\odot$ en $\Opa_{S}$ actuando sobre $\x=(x_{1},x_{2},...,x_{k})$, con 
$S=\hurw_{R}^{k}[[\x]]$, de la siguiente manera
\begin{equation}
\Opa_{\phi}(\x)\odot\Opa_{\psi}(\x)=\left(C_{n}(\phi,\psi)\right)_{n\in\N}
\end{equation}
con
\begin{eqnarray}
C_{0}(\phi,\psi)&=&\x\\
C_{n}(\phi,\psi)&=&\sum_{l=0}^{n-1}\binom{n-1}{l}\mathfrak{D}_{\phi\cdot\psi}^{l}(\mathfrak{D}_{\phi}(\x))\cdot\mathfrak{D}_{\phi\cdot\psi}^{n-1-l}(\mathfrak{D}_{\psi}(\x)), \ \ n\geq1
\end{eqnarray}
\end{definition}

\begin{lemma}\label{lemma_ope_prod}
Para todo $\phi,\psi$ en $\hurw_{R}^{k}[[\x]]$ se tiene 
\begin{equation}
\Opa_{\phi}(\x)\odot\Opa_{\psi}(\x)=\Opa_{\phi\cdot\psi}(\x).
\end{equation}
\end{lemma}
\begin{proof}
Queremos calcular las primeras derivadas direccionales de $\x$ a lo largo de $\phi\cdot\psi$. 
Tenemos
\begin{eqnarray*}
\mathfrak{D}_{\phi\cdot\psi}(\x)&=&\frac{\partial\x}{\partial\x}(\phi\cdot\psi)=\phi\cdot\psi=\mathfrak{D}_{\phi}(\x)\cdot\mathfrak{D}_{\psi}(\x)\\
\mathfrak{D}_{\phi\cdot\psi}^{2}(\x)&=&\mathfrak{D}_{\phi\cdot\psi}(\x)(\mathfrak{D}_{\phi}(\x)\cdot\mathfrak{D}_{\psi}(\x))\\
&=&\mathfrak{D}_{\phi\cdot\psi}(\x)(\mathfrak{D}_{\phi}(\x))\cdot\mathfrak{D}_{\psi}(\x)+\mathfrak{D}_{\phi}(\x)\cdot\mathfrak{D}_{\phi\cdot\psi}(\x)(\mathfrak{D}_{\psi}(\x))
\end{eqnarray*}
Luego por la f\'ormula de Leibniz para la derivada direccional $n$-\'esima obtenemos
\begin{eqnarray*}
\mathfrak{D}_{\phi\cdot\psi}^{n}(\x)&=&\mathfrak{D}_{\phi\cdot\psi}^{n-1}(\mathfrak{D}_{\phi}(\x)\cdot\mathfrak{D}_{\psi}(\x))\\
&=&\sum_{l=0}^{n-1}\binom{n-1}{l}\mathfrak{D}_{\phi\cdot\psi}^{l}(\mathfrak{D}_{\phi}(\x))\cdot\mathfrak{D}_{\phi\cdot\psi}^{n-1-l}(\mathfrak{D}_{\psi}(\x))
\end{eqnarray*}
y $C_{n}(\phi,\psi)=\mathfrak{D}_{\phi\cdot\psi}^{n}(\x)$.
\end{proof}

Ahora mostraremos que $\Opa_{S}(\x)$ es un anillo

\begin{theorem}
$(\Opa_{S}(\x),\boxplus,\odot)$ es un anillo conmutativo con unidad aditiva 
$\textbf{0}_{\Opa}=(\x,0,0,...)$ y unidad multiplicativa $\textbf{1}_{\Opa}=(\x,\textbf{1},0,0,...)$. 
Este anillo ser\'a llamado \textbf{anillo aut\'onomo} de mapas $\Opa_{\phi}$.
\end{theorem}
\begin{proof}
Ya fue mostrado que $(\Opa_{S}(\x),\boxplus)$ es un grupo conmutativo. Por el Lemma \ref{lemma_ope_prod}, $\Opa_{\phi}(\x)\odot\Opa_{\psi}(\x)$ est\'a contenido
en $\Opa_{S}(\x)$. Adem\'as $\textbf{1}_{\Opa}$ es una unidad con respecto a $\odot$
en $\Opa_{S}(\x)$, pues $\Opa_{\textbf{1}}(\x)=\textbf{1}_{\Opa}$ y
\begin{eqnarray*}
\textbf{1}_{\Opa}\odot\Opa_{\phi}(\x)&=&\Opa_{\textbf{1}}(\x)\odot\Opa_{\phi}(\x)\\
&=&\Opa_{\textbf{1}\cdot\phi}(\x)\\
&=&\Opa_{\phi}(\x).
\end{eqnarray*}
La asociatividad sigue de la asociatividad de $\cdot$ en $\hurw_{R}^{k}[[\x]]$. Finalmente 
\begin{eqnarray*}
\Opa_{\varphi}(\x)\odot(\Opa_{\phi}(\x)\boxplus\Opa_{\psi}(\x))&=&\Opa_{\varphi}(\x)\odot\Opa_{\phi+\psi}(\x)\\
&=&\Opa_{\varphi\cdot(\phi+\psi)}(\x)\\
&=&\Opa_{\varphi\cdot\phi+\varphi\cdot\psi}(\x)\\
&=&\Opa_{\varphi\cdot\phi}(\x)\boxplus\Opa_{\varphi\cdot\psi}(\x)\\
&=&[\Opa_{\varphi}(\x)\odot\Opa_{\phi}(\x)]\boxplus[\Opa_{\varphi}(\x)\odot\Opa_{\psi}(\x)]
\end{eqnarray*}
Luego $(\Opa_{S}(\x),\boxplus,\odot)$ es un anillo con unidad $\textbf{1}_{\Opa}$ y divisores de cero.
\end{proof}

\begin{theorem}
$\Opa_{S}(\x)$ es un $\exp(R)$-\'algebra.
\end{theorem}
\begin{proof}
Sigue directamente del Teorema \ref{theo_exp_modulo} y de la Proposici\'on \ref{prop_exp}.
\end{proof}

\section{Flujo $k$-dimensional}

\begin{definition}
Definimos un \textbf{sistema din\'amico} $k$-dimensional sobre el anillo $R$ como la terna 
$(R,\hurw_{S}[[t]],\Phi)$ donde $R$ es el \textbf{conjunto de tiempos}, $\hurw_{S}[[t]]$ el \textbf{espacio de fases} y 
$\Phi$ es el mapa $\Phi:R\times\hurw_{S}[[t]]\times\hurw_{R}^{k}[[\x]]\rightarrow\hurw_{S}[[t]]$, con $S=\hurw_{R}^{k}[[\x]]$, definido por
\begin{equation}\label{flujo_diferencial}
\Phi(t,\x,F(\x))=\x+\sum_{n=1}^{\infty}\mathfrak{D}_{F(\x)}^{n}(\x)\frac{t^{n}}{n!}.
\end{equation}
\end{definition}

\begin{theorem}
$\Phi(t,\x,F(\x))$ es un flujo.
\end{theorem}
\begin{proof}
Cuando $t=0$ tenemos de (\ref{flujo_diferencial}) que $\Phi(0,\x,F(\x))=\x$. Ahora probaremos
la propiedad 2) de un flujo. Por un lado, la expansión de Taylor de $\Phi(s+t,\x,F(\x))$ es
$\sum_{n=0}^{\infty}\mathfrak{D}_{F(\Phi_{s})}^{n}\Phi(s,\x,F(\x))\frac{t^{n}}{n!}$. Por otro lado, 
haciendo $\Phi_{s}=\Phi(s,\x,F(\x))$ tenemos
\begin{eqnarray*}
\Phi(t,\Phi(s,\x,F(\x)),F(\x))
&=&\Phi_{s}+\sum_{n=1}^{\infty}\mathfrak{D}_{F(\Phi_{s})}^{n}\Phi(s,\x,F(\x))\frac{t^{n}}{n!}\\
&=&\sum_{n=0}^{\infty}\mathfrak{D}_{F(\Phi_{s})}^{n}\Phi(s,\x,F(\x))\frac{t^{n}}{n!}\\
&=&\Phi(s+t,\x,F(\x)).
\end{eqnarray*}
como queríamos probar.
\end{proof}

Claramente $\{\Phi_{t}:t\in R\}$ es un grupo 
actuando sobre $\hurw_{S}[[t]]$. 

\begin{theorem}
Tome $F(\x)\in\hurw_{R}^{k}[[\x]]$. Entonces
\begin{equation}
\mathfrak{D}_{F(\x)}(\Phi(t,\x,F(\x)))=\delta_{t}\Phi(t,\x,F(\x))=F(\Phi).
\end{equation}
\end{theorem}
\begin{proof}
De (\ref{flujo_diferencial}) tenemos
\begin{eqnarray*}
\frac{\partial}{\partial\x}\Phi(t,\x,F(\x))F(\x)&=&\frac{\partial}{\partial\x}\left(\x+\sum_{n=1}^{\infty}\mathfrak{D}_{F(\x)}^{n}(\x)\dfrac{t^{n}}{n!}\right)F(\x)\\
&=&F(\x)+\sum_{n=1}^{\infty}\frac{\partial}{\partial\x}\mathfrak{D}_{F(\x)}^{n}(\x)F(\x)\dfrac{t^{n}}{n!}\\
&=&F(\x)+\sum_{n=1}^{\infty}\mathfrak{D}_{F(\x)}^{n+1}(\x)\dfrac{t^{n}}{n!}\\
&=&\sum_{n=0}^{\infty}\mathfrak{D}_{F(\x)}^{n+1}(\x)\dfrac{t^{n}}{n!}\\
&=&\delta_{t}\Phi(t,\x,F(\x)).
\end{eqnarray*}
\end{proof}

A continuaci\'on probaremos que los flujos de $\delta_{t}\Phi=F(\Phi)$ y $\delta_{t}\Phi=aF(\Phi)$,
con $a\in R$, tienen trayectorias que concuerdan. Es decir, 

\begin{theorem}
Para todo $a\in R$ se cumple que $\Phi(t,\x,aF(\x))=\Phi(at,\x,F(\x))$.
\end{theorem}
\begin{proof}
De la Proposici\'on \ref{prop_exp} sabemos que $\Opa_{aF(\x)}(\x)=\exp(a)\cdot\Opa_{F(\x)}(\x)$ para todo 
$a\in R$. Para $a\neq0$ tenemos
\begin{eqnarray*}
\Phi(t,\x,aF(\x))&=&\x+\sum_{n=1}^{\infty}a^{n}\mathfrak{D}_{F(\x)}^{n}(\x)\frac{t^{n}}{n!}\\
&=&\x+\sum_{n=1}^{\infty}\mathfrak{D}_{F(\x)}^{n}(\x)\frac{(at)^{n}}{n!}\\
&=&\Phi(at,\x,F(\x)),
\end{eqnarray*}
Cuando $a=0$, $\Phi(t,\x,\textbf{0})=\Phi(0,\x,F(\x))=\x$
\end{proof}

El teorema anterior implica que existe una correspondencia 1-1 entre las clases
$\exp(R)\cdot\Opa_{F(\x)}(\x)$ y el flujo $\Phi(R,\x,F(\x))$. Esto es cierto porque
\begin{equation*}
\Phi(t,\x,RF(\x))=\Phi(tR,\x,F(\x))=\Phi(R,\x,F(\x))    
\end{equation*}
Sea $I$ un ideal en $R$. Entonces    
\begin{equation*}
\Phi(I,\x,RF(\x))=\Phi(I,\x,F(\x))    
\end{equation*}
Ahora fije $\x$ en $R^{k}$ y $F(\x)$ en $\hurw_{R}^{k}[[\x]]$. Defina el homomorfismo de grupos $\sigma:R\rightarrow\Phi(R,\x,F(\x))$ por by $\sigma(a)=\Phi(a,\x,F(\x))$. Queremos extender el homomorfismo $\sigma$ a un homomorfismo de $R$-módulos mostrando que $\Phi(R,\x,F(\x))$ tiene precisamente estructura de $R$-módulo.
Por la propiedad de grupo de $\Phi(R,\x,F(\x))$, para todo $n\in\N$  tenemos
$$\Phi(nt,\x,F(\x))=\Phi(R,\x,F(\x))\circ\Phi(R,\x,F(\x))\circ\cdots\circ\Phi(R,\x,F(\x))$$
Entonces componemos el flujo $\Phi(t,\x,F(\x))$ consigo mismo $n$ veces y podemos definir
$$n\star\Phi(t,\x,F(\x))=\Phi(t,\x,F(\x))\circ\Phi(t,\x,F(\x))\circ\cdots\circ\Phi(t,\x,F(\x)).$$

De esta manera podemos definir el producto $\star:R\times\Phi(R,\x,F(\x))\rightarrow\Phi(R,\x,F(\x))$ por
$$a\star\Phi(t,\x,F(\x))=\Phi(at,\x,F(\x))$$ 
para todo $a\in R$ y es muy fácil observar que con el producto $\star$ el subgrupo $\Phi(R,\x,F(\x))$ adopta estructura de $R$-módulo y $\sigma$ se convierte en un homomorfismo de $R$-módulo, ya que el anillo $R$ es un $R$-módulo en sí mismo. Ahora queremos entender bajo qué condiciones del elemento fijo $x$ el homomorfismo $\sigma$ se convierte en un isomorfismo de $R$-módulo.

Ahora fije $F(\x)=(f_{1}(\x),...,f_{k}(\x))$, con cada $f_{i}(\x)$ un polinomio en $R[\x]$, y defina el mapa $\epsilon_{\mathbf{a}}$, $\mathbf{a}\in R^{k}$, actuando sobre el flow $\Phi(R,\x,F(\x))$ por
\begin{equation}
\epsilon_{a}\Phi(t,\x,F(\x))=\epsilon_{a}\left(\x+\sum_{n=1}^{\infty}\mathfrak{D}_{F(\x)}^{n}(\x)\frac{t^{n}}{n!}\right)=\mathbf{a}+\sum_{n=1}^{\infty}\mathfrak{D}_{F(\mathbf{a})}^{n}(\mathbf{a})\frac{t^{n}}{n!}
\end{equation}
y sea
\begin{equation}
\Gamma_{a}:=\{\epsilon_{\mathbf{a}}\Phi(t,\x,F(\x)):t\in R\}
\end{equation}
la orbita o trayectoria de $\mathbf{a}$. 
Si $\Gamma_{\x_{0}}={\x_{0}}$, entonces $\x_{0}$ es un punto de equilibrio para $\Phi(R,\x,F(\x))$. Los puntos de equilibrio se obtienen cuando $\mathfrak{D}_{F(\x_{0})}^{n}(\x_{0})=\mathbf{0}$ para $n\geq1$, es decir, cuando
$F(\x_{0})=0$ para algún $\x_{0}\in R^{k}$. Si $\x_{0}$ no es un punto de equilibrio, entonces se llamará punto regular de $\Phi(R,\x,F(\x))$. 
Denote
\begin{eqnarray*}
\ker(\sigma)&=&\{t\in R:\sigma(t)=\x\}\\
&=&\{t\in R:\Phi(t,\x,F(\x))=\x\}
\end{eqnarray*} 
el núcleo de $\sigma$. Si $\overline{\x}$ es un punto de equilibrio, entonces
$\sigma(R)=\overline{\x}$ y $\ker(\sigma)=R$. Si $\x$ es un punto regular, $\ker(\sigma)=\{0\}$ y 
$\sigma$ es un mapa inyectivo. Como $\sigma$ es sobreyectivo, entonces $\sigma$ es un isomorfismo de $R$-módulos. Diremos que el módulo $\Phi(R,\x,F(\x))$ es un módulo trivial si $\x$ es un punto de equilibrio. En caso contrario se llamaría módulo no trivial.

Sea $I$ un ideal de $R$ y $\x$ un punto regular. Entonces $\sigma(I)=\Phi(I,\x,F(\x))$ es un submódulo en $\Phi(R,\x,F(\x))$. Entonces podemos establecer la siguiente correspondencia
\begin{eqnarray}
\{\text{Ideals $I$ in }R\}\Leftrightarrow\{\text{Submodules }\Phi_{I} \text{ of }\Phi(R,\x,F(\x))\}\Leftrightarrow\{\mathbf{u}^{\prime}=IF(\mathbf{u})\}
\end{eqnarray}
donde $\{\mathbf{u}^{\prime}=IF(\mathbf{u})\}$ denota el conjunto de todo los sistemas dinámico de la forma $\mathbf{u}^{\prime}=aF(\mathbf{u})$, con $a\in I$.

Terminamos esta sección con el siguiente resultado
\begin{theorem}
Un $R$-módulo no trivial $\Phi(R,\x,F(\x))$ es un módulo cíclico sin torsión.
\end{theorem}
\begin{proof}
Supongamos que $\x$ es un punto regular. Para demostrar que $\Phi(R,\x,F(\x))$ es un módulo cíclico, basta con 
tomar una unidad $u$ en $R$. Entonces $R\star\Phi(u,\x,F(\x))=\Phi(uR,\x,F(\x))=\Phi(R,\x,F(\x))$ y así 
$\Phi(R,\x,F(\x))$ es cíclico. Ahora demostraremos que $\Phi(R,\x,F(\x))$ es libre de torsión.
Por un lado, existe un ideal $I\subset R$ tal que $\Phi(R,\x,F(\x))$ es isomorfo a $R/I$. Como ya se demostró que $R$ y $\Phi(R,\x,F(\x))$ son isomorfos, entonces se deduce que $I$ es el ideal cero. Por otro lado, que $\Phi(R,\x,F(\x))$ sea un $R$-módulo cíclico es equivalente a decir que el homomorfismo de multiplicación $\tau_{s}:R\rightarrow\Phi(R,\x,F(\x))$, 
$\tau_{s}(a)=a\star\Phi(s,\x,F(\x))$, es un homomorfismo de $R$-módulos sobreyectivos. Escriba $I=\ker(\tau_{s})$. Por el primer teorema de isomorfismo para módulo, $\overline{\tau_{s}}$ es 
un isomorfismo de $R/I$ a $\Phi(R,\x,F(\x))$. Como $\overline{\tau_{s}}$ es el aniquilador 
$\ann(\Phi(s,\x,F(\x)))$ de $\Phi(s,\x,F(\x))$, entonces $\ann(\Phi(s,\x,F(\x)))=I=0$ para cualquier 
$\Phi(s,\x,F(\x))\in\Phi(R,\x,F(\x))$ y $0$ es el único elemento de torsión en $\Phi(R,\x,F(\x))$.
\end{proof}

\section{El anillo $(\rho_{t}\Opa_{S}(\x),\boxplus,\odot)$, $S=\hurw_{R}^{k}[[\x]]$}

Fije $\phi$ en $\hurw_{R}^{k}[[\x]]$. Defina el mapa $\rho_{t}:\Opa_{\phi}(\hurw_{R}^{k}[[\x]])\rightarrow\hurw_{S}[[t]]$, donde $S=\hurw_{R}^{k}[[\x]]$,
como 
\begin{equation}
\rho_{t}\Opa_{\phi}(F(\x))=F(\x)+\sum_{n=1}^{\infty}\mathfrak{D}_{\phi}^{n}(F(\x))\frac{t^{n}}{n!}.
\end{equation}
Entonces
\begin{eqnarray}
\Phi(t,\x,\phi(\x))&=&\rho_{t}\Opa_{\phi}(\x).
\end{eqnarray}

\begin{theorem}
Fije $\phi$ en $\hurw_{R}^{k}[[\x]]$. Entonces $\rho_{t}\Opa_{\phi}(\hurw_{R}^{k}[[\x]])$ es un 
$R$-\'algebra con suma y producto ordinario de series de potencias.
\end{theorem}
\begin{proof}
Tome $F(\x)$ y $G(\x)$ en $\hurw_{R}^{k}[[\x]]$. Como ya fue mostrado que 
$(\Opa_{\phi}(\hurw_{R}^{k}[[\x]]),+,\ast)$ es un $R$-\'algebra, entonces es f\'acil mostrar que
\begin{eqnarray*}
\rho_{t}[\Opa_{\phi}(F(\x))\ast\Opa_{\phi}(G(\x))]&=&\rho_{t}\Opa_{\phi}(F(\x))\cdot\rho_{t}\Opa_{\phi}(G(\x)),\\
\rho_{t}[\Opa_{\phi}(F(\x))+\Opa_{\phi}(G(\x))]&=&\rho_{t}\Opa_{\phi}(F(\x))+\rho_{t}\Opa_{\phi}(G(\x)),\\
\rho_{t}[a\Opa_{\phi}(F(\x))]&=&a\rho_{t}\Opa_{\phi}(F(\x)).
\end{eqnarray*}
Lo afirmado sigue de lo anterior.
\end{proof}

\begin{theorem}
Fije $\phi$ en $\hurw_{R}^{k}[[\x]]$. Entonces
\begin{equation}
\rho_{t}\Opa_{\phi}(F(\x))=F(\rho_{t}\Opa_{\phi}(\x)).
\end{equation}
Este resultado es conocido como la serie de Lie-Groebner-Taylor.
\end{theorem}
\begin{proof}
Es f\'acil mostrar que
\begin{equation}
\rho_{t}\Opa_{\phi}(\textbf{1}x_{i}^{n_{i}})=[\textbf{1}x_{i}\circ\rho_{t}\Opa_{\phi}(\x)]^{n_{i}}
\end{equation}
en donde $\circ$ indica composici\'on. Luego en general
\begin{equation}
\rho_{t}\Opa_{\phi}(\textbf{1}x_{1}^{n_{1}}\cdots x_{k}^{n_{k}})=[\textbf{1}x_{1}\circ\rho_{t}\Opa_{\phi}(\x)]^{n_{1}}\cdots[\textbf{1}x_{k}\circ\rho_{t}\Opa_{\phi}(\x)]^{n_{k}}
\end{equation}
y
\begin{equation*}
\rho_{t}\Opa_{\phi}(F(\x))=F(\rho_{t}\Opa_{\phi}(\x)).
\end{equation*}
\end{proof}

Ahora tomemos $F(\x)=\x$ en la definici\'on del mapa $\rho_{t}$. Obtendremos el anillo isomorfo a $(\Opa_{T}(\x),\boxplus,\odot)$, con $T=\hurw_{R}^{k}[[\x]]$. Tenemos

\begin{definition}
Tome $\phi(\x)$ y $\psi(\x)$ de $\hurw_{R}^{k}[[\x]]$. Definimos la suma $\boxplus$ y el producto
$\odot$ en $\rho_{t}\Opa_{T}(\x)$ así
\begin{eqnarray}
\rho_{t}\Opa_{\phi(\x)}(\x)\boxplus\rho_{t}\Opa_{\psi(\x)}(\x)&=&\rho_{t}[\Opa_{\phi(\x)}(\x)\boxplus\Opa_{\psi(\x)}(\x)]\\
\rho_{t}\Opa_{\phi(\x)}(\x)\odot\rho_{t}\Opa_{\psi(\x)}(\x)&=&\rho_{t}[\Opa_{\phi(\x)}(\x)\odot\Opa_{\psi(\x)}(\x)]
\end{eqnarray}
\end{definition}

Luego tenemos el siguiente resultado

\begin{theorem}
El conjunto $\rho_{t}\Opa_{T}(\x)$ con la suma $\boxplus$ y el producto $\odot$ es un anillo conmutativo con unidades $\rho_{t}\Opa_{\textbf{0}}(\x)=\x$ y $\rho_{t}\Opa_{\textbf{1}}(\x)=\x+\textbf{1}t$. El anillo $\rho_{t}\Opa_{T}(\x)$ ser\'a llamado \textbf{anillo de flujos k-dimensionales}.
\end{theorem}
\begin{proof}
Por la definición arriba el conjunto $\rho_{t}\Opa_{T}(\x)$ hereda las propiedades de 
anillo de $(\Opa_{T}(\x),\boxplus,\odot)$. Por otro lado, $\rho_{t}\Opa_{\textbf{0}}(\x)=\rho_{t}\textbf{0}_{\Opa}=\x$ y $\rho_{t}\Opa_{\textbf{1}}(\x)=\rho_{t}\textbf{1}_{\Opa}=\x+\textbf{1}t$.
\end{proof}

La raz\'on para construir el anillo $\rho_{t}\Opa_{T}(\x)$ es porque este contiene todas 
las soluciones de los sistemas din\'amicos $\delta_{t}\Phi=F(\Phi)$ para cada funci\'on $F(\x)$ en
$\hurw_{R}^{k}[[\x]]$. En este anillo es posible descomponer las soluciones de sistemas din\'amicos 
$k$-dimensionales en soluciones m\'as simples.\\ 
Primero suponga que $F(\x)=F_{1}(\x)+F_{2}(\x)+\cdots+F_{n}(\x)$. Deseamos resolver la ecuaci\'on diferencial $\delta_{t}\Phi=F(\Phi)$. El flujo de esta ecuaci\'on viene a ser
\begin{eqnarray*}
\Phi(t,\x,F(\x))&=&\rho_{t}\Opa_{F(\x)}(\x)\\
&=&\rho_{t}\Opa_{F_{1}(\x)+\cdots+F_{n}(\x)}(\x)\\
&=&\rho_{t}\left(\bbox_{i=1}^{n}\Opa_{F_{i}(\x)}(\x)\right)\\
&=&\bbox_{i=1}^{n}\rho_{t}\Opa_{F_{i}(\x)}(\x)\\
&=&\bbox_{i=1}^{n}\Phi(t,\x,F_{i}(\x)).
\end{eqnarray*} 
Así el flujo de $\delta_{t}\Phi=F(\Phi)$ se descompone en sumandos en donde cada sumando es
el flujo de las ecuaciones $\delta_{t}\Phi=F_{i}(\Phi)$.

Ahora suponga que $F(\x)$ factoriza como
$F(\x)=F_{1}(\x)\cdot F_{2}(\x)\cdots F_{n}(\x)$ en $\hurw_{R}^{k}[[\x]]$ y busquemos la soluci\'on
a la ecuaci\'on $u^{\prime}=F(u)$. Entonces su flujo es
\begin{eqnarray*}
\Phi(t,\x,F(\x))&=&\rho_{t}\Opa_{F(\x)}(\x)\\
&=&\rho_{t}\Opa_{F_{1}(\x)\cdots F_{n}(\x)}(\x)\\
&=&\rho_{t}\left(\bcast_{i=1}^{n}\Opa_{F_{i}(\x)}(\x)\right)\\
&=&\bcast_{i=1}^{n}\rho_{t}\Opa_{F_{i}(\x)}(\x)\\
&=&\bcast_{i=1}^{n}\Phi(t,\x,F_{i}(\x)).
\end{eqnarray*}
en donde cada $\Phi(t,\x,F_{i}(\x))$ es el flujo de la ecuaci\'on $u^{\prime}=F_{i}(u)$.

\begin{example}
\textbf{Sistema de Ecuaciones Diferenciales Lineales.} Un sistema de ecuaciones diferenciales lineal es uno de la forma
\begin{equation}\label{eqn_lineal}
    \x^{\prime}=A\x
\end{equation}
en donde $A=(a_{ij})$ es una matriz $n\times n$ con entradas en $\R$ y $\x=(x_{1},...,x_{n})^{\top}\in\R^{n}$. Por la ecuación (\ref{eqn_oper_auto})
\begin{eqnarray*}
\Opa_{A\x}(\x)=(\x,A\x,A^{2}\x,A^{3}\x,...)
\end{eqnarray*}
y el flujo de (\ref{eqn_lineal}) es
\begin{eqnarray*}
\Phi(t,\x,A\x)&=&\x+\sum_{n=1}^{\infty}A^{n}\x\frac{t^{n}}{n!}\\
&=&\left(1+\sum_{n=1}^{\infty}A^{n}\frac{t^{n}}{n!}\right)\x\\
&=&e^{At}\x
\end{eqnarray*}
siendo este resultado compatible con la solución ya conocida de (\ref{eqn_lineal}).
\end{example}

\begin{theorem}
Sea $F(\x)$ la función
\begin{eqnarray}
F(\x)=
\left(
\begin{array}{c}
    \prod_{j=1}^{m}\left(\sum_{i=1}^{n}a_{1,i,j}x_{i}+b_{1,j}\right)\\
    \prod_{j=1}^{m}\left(\sum_{i=1}^{n}a_{2,i,j}x_{i}+b_{2,j}\right)\\
    \vdots\\
    \prod_{j=1}^{m}\left(\sum_{i=1}^{n}a_{n,i,j}x_{i}+b_{n,j}\right)
\end{array}
\right)
\end{eqnarray}
con los $a_{l,i,j}\in R$, entonces la ecuación $\x^{\prime}=F(\x)$ tiene solución
\begin{equation}
    \Phi(t,\x,F(\x))=\bcast_{j=1}^{n}\left(e^{tA_{j}}\x+\int_{0}^{t} e^{(t-s)A_{j}}\textbf{b}_{j}ds\right).
\end{equation}
\end{theorem}
\begin{proof}
Sea $F_{j}(\x)=(\sum_{i=1}^{n}a_{1,i,j}x_{i},...,\sum_{i=1}^{n}a_{n,i,j}x_{i})^{\top}+(b_{1,j},...,b_{n,j})$. Entonces $F(\x)=F_{1}(\x)\cdot F_{2}(\x)\cdots F_{m}(\x)$. Además como $F_{j}(\x)=A_{j}\x+\textbf{b}_{j}$, con $A_{j}=(a_{l,i,j})_{l,i=1}^{n}$ y $\textbf{b}_{j}=(b_{1,j},...,b_{n,j})$,
entonces resolviendo el sistema lineal no homogéneo $\x^{\prime}=A_{j}\x+\textbf{b}_{j}$ y luego multiplicando en el anillo $\rho_{t}\Opa_{T}(\x)$, con $T=\hurw_{R}^{n}[[\x]]$, obtenemos el resultado
deseado.
\end{proof}

Este teorema es muy útil para encontrar la solución exacta del sistema planar
\begin{eqnarray*}
x^{\prime}&=&p(x,y)\\
y^{\prime}&=&q(x,y)
\end{eqnarray*}
con $p(x,y), q(x,y)$ polinomios irreducibles en $\R[x,y]$ el anillo de polinomios reales en $x,y$.

\begin{example}
\textbf{Ecuación Lotka-Volterra.} Esta es una ecuación de la forma
\begin{eqnarray*}
x^{\prime}&=&x(a-bx-cy)\\
y^{\prime}&=&y(d-ex-fy)
\end{eqnarray*}
En notación vectorial esta ecuación se escribe como
\begin{eqnarray}\label{eqn_LV_vec}
\frac{d}{dt}
\left(
\begin{array}{c}
     x \\
     y 
\end{array}
\right)=
\left(
\begin{array}{c}
     x \\
     y 
\end{array}
\right)\cdot
\left(
\begin{array}{c}
     a-bx-cy \\
     d-ex-fy 
\end{array}
\right)
\end{eqnarray}
Deseamos encontrar una solución en el anillo $\rho_{t}\Opa_{T}(x,y)$ con $T=H_{\R}^{2}[[x,y]]$. La solución de (\ref{eqn_LV_vec}) la encontraremos por
encontrar las soluciones de las ecuaciones
\begin{equation*}
    \frac{d}{dt}\left(
    \begin{array}{c}
         x\\
         y
    \end{array}
    \right)=
    \left(
    \begin{array}{c}
         x\\
         y
    \end{array}
    \right),\ \ \ 
    \frac{d}{dt}\left(
    \begin{array}{c}
         x\\
         y
    \end{array}
    \right)=
    \left(
    \begin{array}{c}
         a-bx-cy\\
         d-ex-fy
    \end{array}
    \right)
\end{equation*}
y juntar todo en $\rho_{t}\Opa_{T}(x,y)$. Entonces la solución de la ecuación (\ref{eqn_LV_vec}) es
\begin{equation}\label{eqn_sol_LV}
    \left(
    \begin{array}{c}
         x(t) \\
         y(t) 
    \end{array}
    \right)=
    e^{tI}
    \left(
    \begin{array}{c}
         x_{0} \\
         y_{0} 
    \end{array}
    \right)\odot
    \left(
    e^{-tB}
    \left(
    \begin{array}{c}
         x_{0} \\
         y_{0} 
    \end{array}
    \right)
    +\int_{0}^{t}e^{-(t-s)B}\textbf{b}ds
    \right)
\end{equation}
en donde $I=\left(\begin{array}{cc}
     1&0  \\
     0&1 
\end{array}\right)$ es la matriz identidad,
$B=\left(\begin{array}{cc}
     b&c  \\
     e&f 
\end{array}\right)$ y $\textbf{b}=(a,t)^{\top}$. Finalmente resolviendo el segundo factor de (\ref{eqn_sol_LV}) por cualquier método tradicional obtenemos
\begin{equation}\label{eqn_sol_final_LV}
    \left(
    \begin{array}{c}
         x(t) \\
         y(t) 
    \end{array}
    \right)=
    e^{tI}
    \left(
    \begin{array}{c}
         x_{0} \\
         y_{0} 
    \end{array}
    \right)\odot
    \left(
    e^{-tB}
    \left(
    \begin{array}{c}
         x_{0} \\
         y_{0} 
    \end{array}
    \right)
    +
    \left(
    \begin{array}{c}
         u \\
         v 
    \end{array}
    \right)
    \right)
\end{equation}
en donde $u=\frac{af-cd}{bf-ce}$ y $v=\frac{bd-ae}{bf-ce}$ probado que
$bf-ce\neq0$. Cuando hacemos $t=0$ en (\ref{eqn_sol_final_LV})
\begin{equation*}
    \left(
    \begin{array}{c}
         x(0) \\
         y(0) 
    \end{array}
    \right)=
    \left(
    \begin{array}{c}
         x_{0} \\
         y_{0} 
    \end{array}
    \right)\odot
    \left(
    \left(
    \begin{array}{c}
         x_{0} \\
         y_{0} 
    \end{array}
    \right)
    +
    \left(
    \begin{array}{c}
         u \\
         v 
    \end{array}
    \right)
    \right)=
    \left(
    \begin{array}{c}
         x_{0} \\
         y_{0} 
    \end{array}
    \right)
\end{equation*}
ya que $(x_{0},y_{0})^{\top}$ es el elemento neutro en  $\rho_{t}\Opa_{T}(x,y)$ con $T=\hurw_{\R}^{2}[[x,y]]$.
\end{example}

\begin{example}
\textbf{Ecuación de Van der Pol}. Esta ecuaci\'on es de la forma
\begin{equation}
x^{\prime\prime}+\mu(x^{2}-1)x^{\prime}+x=0.
\end{equation}
En notaci\'on vectorial esta ecuaci\'on tiene la forma
\begin{equation}\label{eqn_van}
\left(
\begin{array}{c}
x\\
y
\end{array}
\right)^{\prime}
=\left(
\begin{array}{c}
y\\
-x-\mu(x^{2}-1)y
\end{array}
\right)
\end{equation}
Queremos encontrar una soluci\'on en el anillo $\rho_{t}\Opa_{T}(x,y)$ con $T=\hurw_{\R}^{2}[[x,y]]$. 
La ecuaci\'on (\ref{eqn_van}) se puede poner en la forma de factores
\begin{equation}
\left(
\begin{array}{c}
x\\
y
\end{array}
\right)^{\prime}
=
\left(
\begin{array}{c}
y\\
-x
\end{array}
\right)
+
\left(
\begin{array}{c}
0\\
-\mu
\end{array}
\right)
\cdot
\left(
\begin{array}{c}
0\\
x-1
\end{array}
\right)
\cdot
\left(
\begin{array}{c}
0\\
x+1
\end{array}
\right)
\cdot
\left(
\begin{array}{c}
0\\
y
\end{array}
\right)
\end{equation}
Luego debemos resolver las ecuaciones
\begin{equation*}
\left(
\begin{array}{c}
x\\
y
\end{array}
\right)^{\prime}
=
\left(
\begin{array}{c}
y\\
-x
\end{array}
\right)
,\ \ \
\left(
\begin{array}{c}
x\\
y
\end{array}
\right)^{\prime}
=
\left(
\begin{array}{c}
0\\
-\mu
\end{array}
\right)
,\ \ \
\left(
\begin{array}{c}
x\\
y
\end{array}
\right)^{\prime}
=
\left(
\begin{array}{c}
0\\
x-1
\end{array}
\right)
\end{equation*}
\begin{equation*}
\left(
\begin{array}{c}
x\\
y
\end{array}
\right)^{\prime}
=
\left(
\begin{array}{c}
0\\
x+1
\end{array}
\right)
,\ \ \
\left(
\begin{array}{c}
x\\
y
\end{array}
\right)^{\prime}
=
\left(
\begin{array}{c}
0\\
y
\end{array}
\right)
\end{equation*}
y juntar todo en el anillo $\rho_{t}\Opa_{T}(x,y)$.
Entonces la soluci\'on de la ecuaci\'on de Van der Pol es
\begin{eqnarray}\label{eqn_sol_van}
\left(
\begin{array}{c}
x(t)\\
y(t)
\end{array}
\right)
&=&
\left(
\begin{array}{c}
x_{0}\cos t-y_{0}\sin t\\
x_{0}\sin t+y_{0}\cos t
\end{array}
\right)
\boxplus
\left(
\begin{array}{c}
x_{0}\\
-\mu t+y_{0}
\end{array}
\right)
\odot
\left(
\begin{array}{c}
x_{0}\\
(x_{0}-1)t+y_{0}
\end{array}
\right)\nonumber\\
&&\odot
\left(
\begin{array}{c}
x_{0}\\
(x_{0}+1)t+y_{0}
\end{array}
\right)
\odot
\left(
\begin{array}{c}
x_{0}\\
y_{0}e^{t}
\end{array}
\right)
\end{eqnarray}
para todo $x_{0},y_{0}\in\R$. Cuando hacemos $t=0$ en (\ref{eqn_sol_van}) obtenemos
\begin{eqnarray}
\left(
\begin{array}{c}
x(0)\\
y(0)
\end{array}
\right)
&=&
\left(
\begin{array}{c}
x_{0}\\
y_{0}
\end{array}
\right)
\boxplus
\left(
\begin{array}{c}
x_{0}\\
y_{0}
\end{array}
\right)
\odot
\left(
\begin{array}{c}
x_{0}\\
y_{0}
\end{array}
\right)\nonumber\\
&&\odot
\left(
\begin{array}{c}
x_{0}\\
y_{0}
\end{array}
\right)
\odot
\left(
\begin{array}{c}
x_{0}\\
y_{0}
\end{array}
\right)\\
&=&
\left(
\begin{array}{c}
x_{0}\\
y_{0}
\end{array}
\right)\boxplus
\left(
\begin{array}{c}
x_{0}\\
y_{0}
\end{array}
\right)=\left(
\begin{array}{c}
x_{0}\\
y_{0}
\end{array}
\right)
\end{eqnarray}
en donde tuvimos presente que $(x_{0},y_{0})$ es el elemento neutro en 
$\rho_{t}\Opa_{T}(x,y)$.
Cuando $\mu=0$ en (\ref{eqn_sol_van}) obtenemos la soluci\'on
\begin{equation}
\left(
\begin{array}{c}
x(t)\\
y(t)
\end{array}
\right)=
\left(
\begin{array}{c}
x_{0}\cos t-y_{0}\sin t\\
x_{0}\sin t+y_{0}\cos t
\end{array}
\right)
\end{equation}
de la ecuaci\'on $x^{\prime\prime}+x=0$.

\end{example}

Finalizamos este trabajo mostrando como se relacionan la composición $\circ$ con las operaciones $\boxplus$ y $\odot$
\begin{theorem}
En el anillo $\rho_{t}\Opa_{T}(\x)$ tenemos las siguientes identidades relacionando las operaciones $\circ$, $\boxplus$ y $\odot$
\begin{enumerate}
\item $\Phi_{t}(\x,F(\x)+G(\x))\circ\Phi_{s}(\x,F(\x)+G(\x))=\Phi_{t+s}(\x,F(\x))\boxplus\Phi_{t+s}(\x,G(\x))$.
\item $\Phi_{t}(\x,F(\x)\cdot G(\x))\circ\Phi_{s}(\x,F(\x)\cdot G(\x))=\Phi_{t+s}(\x,F(\x))\odot\Phi_{t+s}(\x,G(\x))$
\end{enumerate} 
\end{theorem}
\begin{proof}
La prueba es directa.
\end{proof}

\end{document}